\documentclass[11pt,a4paper,dvips]{article}
\usepackage{amsthm}
\usepackage{amsmath}
\usepackage{amssymb}
\usepackage{array}
\usepackage{graphicx}
\usepackage{graphics}
\usepackage{epsfig}
\usepackage{amsfonts,amscd}
\usepackage{exscale}
\usepackage{eucal}

\newcommand{\w}{\widetilde}

\newcommand{\ov}{\overline}

\numberwithin{equation}{section}
\newtheorem{thm}{\bf Theorem}[section]

\newtheorem{prop}{\bf Proposition}[section]
\newtheorem{cor}[thm]{\bf Corollary}
\newtheorem{defn}{\bf Definition}[section]
\theoremstyle{remark}
\newtheorem{rem}{\bf Remark}[section]
\newtheorem{exmp}{\bf Example}[section]

\begin{document}
\def \b{\Box}

\begin{center}
{\Large {\bf VECTOR GROUPOIDS}}\\[0.5cm]
{VASILE POPU\c TA and GEORGHE IVAN}\\[0.8cm]
\end{center}

{\small ABSTRACT. The main purpose of this paper is to study the
vector groupoids. This is an algebraic structure which combines
the concepts of Brandt groupoid and vector space such that these
are compatible. The new concept of vector groupoid has
applications in geometry and other areas.}
{}

\section{INTRODUCTION}
\indent\indent A groupoid, also known as a {\it virtual group}
\cite{rams71}, is an algebraic structure introduced by H. Brandt
\cite{brandt}.
 A groupoid (in the sense of Brandt) can be thought as a set with a partially defined
 multiplication, for which the usual properties of a group hold whenever they make sense.

A generalization of Brandt groupoid  has appeared in
\cite{ehre50}. C. Ehresmann added further structures ( topological
and differentiable as well as algebraic) to groupoids.

Groupoids and its generalizations (topological groupoids, Lie
groupoids, measure groupoids, sympectic groupoids etc.) are
mathematical structures that have proved to be useful in many
areas of science [algebraic topology (\cite{brow88},
\cite{dumoiv}), harmonic analysis and operators algebras
(\cite{dumoiv}, \cite{renaul}, \cite{west}), differential geometry
and its applications (\cite{canwei}, \cite{cdw87},
\cite{mackey63}, \cite{ramsren01}, \cite{weigro}), noncommutative
geometry (\cite{cone94}), algebraic and geometric combinatorics
(\cite{john}, \cite{ziv06}), dynamics of networks
(\cite{diaste04}, \cite{golust06}, \cite{stegopi03} and more].

It is remarkable to note that according to A. Connes
\cite{cone94}, Heisenberg was discovered quantum mechanics by
considering the groupoid of quantum transitions rather than the
group of symmetry.

The paper is organized as follows. In Section 2 we define
groupoids and useful properties of them are presented.
 In Section 3 we introduce the concept of vector groupoid and its
 properties are established. In Section 3  we give some algebraic
 constructions of vector groupoids.

\section{BRANDT GROUPOIDS}
\indent\indent We recall the minimal necessary backgrounds on
groupoids for our developments (for further details see e.g.
\cite{brow87}, \cite{higg71}, \cite{ivan02}, \cite{vpop} and
references therein for more details).

\begin{defn}(\cite{cdw87}) A \textbf{groupoid $ G $
over} $ G_{0} $ ( \textbf{in the sense of Brandt} ) is a pair $
(G, G _0) $ of nonempty sets such that $ G _0\subseteq G  $
endowed with two surjective maps $\alpha ,\beta :G \rightarrow G
_0$ ( called {\bf source}, respectively {\bf target}, a partially
binary operation (called {\bf multiplication}) $~m :G
_{(2)}\rightarrow G,~(x,y)\longmapsto m \left( x,y\right) :=x\cdot
y,~$ where  $~G _{(2)}:= G\times_{(\beta, \alpha)}G = \left\{
\left( x,y\right) \in G \times G \mid \beta \left( x\right)
=\alpha \left( y\right) \right\}$ is the {\bf set of composable
pairs} and a map $~\iota :G \rightarrow G,~x\longmapsto \iota
(x):=x^{-1}$ ( called \textbf{inversion}), which verify the
following conditions:

(G)~({\bf associativity}): $~(x\cdot y)\cdot z=x\cdot (y\cdot z)$
in the sense that if one of two products $(x\cdot y)\cdot z$ and
$x\cdot (y\cdot z)$ is defined, then the other product is also
defined and they are equals;

(G2)~({\bf units}): for each $x\in G~ \Rightarrow ~(\alpha(x),x),
~(x,\beta( x))\in G _{(2)}~$ and  we have $~\alpha(x)\cdot
x=x\cdot \beta(x)=x $;

 (G3)~({\bf inverses}): for each $ x\in G~\Rightarrow ~ (x,x^{-1}), ~(x^{-1},x)\in G
_{(2)} $ and we have $~x^{-1}\cdot x=\beta(x), ~~x\cdot
x^{-1}=\alpha(x).$
\end{defn}

A groupoid $G$ over $G _0$ with the \textit{structure functions} $
\alpha ,\beta ,m ,\iota $  is denoted by $(G, \alpha, \beta, m,
\iota, G _0)$ or $(G, \alpha, \beta, G _0)$ or $(G, G _0)$.
 The element $\alpha(x)$ respectively $\beta(x)$ is called the
\textit{left unit} respectively \textit{ right unit} of $x;$
$~G_0$ is called the \textit{ unit set} of $G$. The map $(\alpha,
\beta)$ defined by:
\[(\alpha,\beta):G\rightarrow G _0\times G _0,\quad (\alpha,
\beta)(x):=(\alpha(x), \beta(x)),~x\in G,\] is called the
\textit{anchor map} of $G$. For each $ u\in G_{0} $, the set $
G_{u}:= \alpha^{-1}(u)$ ( resp. $ G_{u}:= \beta^{-1}(u)$ ) is
called \textit{ $\alpha-$ fibre} ( resp. \textit{ $\beta-$ fibre}
) of $ G$ at $ u\in G_{0}$. If $~u,v\in G_{0}~$ we will write
$~G_{v}^{u} = \alpha^{-1}(u) \cap \beta^{-1}(v).~$

A  groupoid $~( G, G_{0})$ is said to be {\it transitive}, if its
 anchor map is surjective.

{\bf Convention.}~(1)~We write sometimes $~x y~$ for $~m(x,y)~$,
if $ (x,y) \in G_{(2)}.$

(2)~Whenever we write a product in a given groupoid, we are
assuming that it is defined.\hfill$\b$

In the following proposition we summarize some basic rules of
algebraic calculation in a Brandt groupoid obtained directly from
definitions.

\begin{prop}
(\cite{ivan02}) {\it In a groupoid $~( G, \alpha, \beta, m, \iota
, G_{0})~$ the following assertions hold :

$(i)~~~\alpha(u)=\beta(u)=u,~~~ u\cdot u=u \quad \hbox{and}\quad
\iota (u)=u,~ \forall u\in G_0;$

$(ii)~\alpha \left( x\cdot y\right) =\alpha \left( x\right) \quad
\hbox{and}\quad \beta \left( x\cdot y\right) =\beta \left(
y\right) ,~ \forall  \left( x,y\right) \in G _{\left( 2\right) };$

$(iii)~\alpha \left( x^{-1}\right) =\beta \left( x\right) \quad
\hbox{and}\quad \beta \left( x^{-1}\right) =\alpha \left( x\right)
,~ \forall  x\in G;$

$(iv)~$ {\bf(cancellation law)} If  for $ x,y_1,y_2,z\in G $ we
have $(x,y_1),(x,y_2),$ $(y_1,z),(y_2,z)\in G_{(2)},$ then:
\[(a) \quad  x\cdot y_1=x\cdot y_2~~\Rightarrow~~ y_1=y_2;~~~
(b) \quad y_1\cdot z=y_2\cdot z~~\Rightarrow~~ y_1=y_2.\]

{\it(v)} For each $ x\in G$  we have $~(x^{-1})^{-1}=x.$

{\it(vi)} If $(x,y)\in G_{(2)},$ then $(y^{-1},x^{-1})\in G_{(2)}$
and the equality holds: \[ (x\cdot y)^{-1}=y^{-1}\cdot x^{-1}.\]

{\it(vii)} For all $(x,y)\in G_{(2)}$, the following equalities
hold: \[ x^{-1}\cdot(x\cdot y)=y ~~~\hbox{and}~~~(x\cdot y)\cdot
y^{-1}=x.\]}
\end{prop}

In a groupoid $(G, G_{0})$ for any $ u\in G_{0}$, the set
$~G(u):=\alpha^{-1}(u)\cap\beta^{-1}(u)=\{ x\in G~|~
\alpha(x)=\beta(x)=u~\}$ is a group under the restriction of the
partial multiplication $~ m~$ to $ G(u),$ called the {\it isotropy
group at} $u$ of $G$.

\begin{prop}
(\cite{ivan02}) {\it Let $~( G, \alpha, \beta, m, \iota , G_{0})~$
be a groupoid. Then:

$(i)~~~~~\alpha \circ \iota = \beta,~~ \beta \circ \iota =
\alpha~~ \hbox{and}~~ \iota \circ \iota = Id_{G}.$

$(ii)~~~\varphi : G(\alpha(x)) \to G(\beta(x)),~ \varphi (z):=
x^{-1} z x $ is an isomorphism of groups.

$(iii)~ $ If $ (G, G_{0} ) $ is transitive, then all isotropy
groups are isomorphes.}
\end{prop}

A {\it group bundle} is a  groupoid $~( G, G_{0})~$ with the
property that  $~\alpha (x) = \beta (x)~$ for all $~x\in G .$
Moreover,a group bundle is the union of its isotropy groups $~G(u)
= \alpha^{-1}(u), u \in G_{0}~$ (here, two elements may be
composed iff they lie in the same fiber $\alpha^{-1}(u)$~).

If $~( G,\alpha,\beta, G_{0})~$ is a groupoid then $~Is(G): = \{
x\in G~|~\alpha (x)=\beta (x) \}~$ is a group bundle, called the
{\it isotropy group bundle}  of $~G.~$

\begin{exmp}
{\it (i)}~ Any group $ G $ having $ e $ as unity, is a groupoid
over $ G_{0} = \{ e \}$ with the structure functions $ \alpha,
\beta, m, \iota $ given by:\\ $\alpha(x) = \beta(x) = e,~ \iota(x)
= x^{-1}$ for all $ x\in G$ and $ m(x,y)= xy $ for all $x,y\in G$.

{\it (ii)}~ Any set $~X~$ can be endowed with a {\it nul groupoid}
structure over itself. For this we take: $~\alpha = \beta =\iota =
Id_{X}; ~x , y\in X~$ are composable iff $~x=y~$ and we define $~
x\cdot x = x.$

{\it (iii)}~ The Cartesian product $~G:= X \times X~$ has a
structure of groupoid over  $ \Delta_{X} = \{ (x,x)\in X\times X
~|~ x\in X \}$ by taking the structure functions as follows:
$~\w{{\alpha}}(x,y):= (x,x),~ \w{{\beta}}(x,y):= (y,y);~$ the
elements $~ (x,y)~$ and $~(y^{\prime},z)~$ are composable in
$~G:=X \times X~$ iff $~y^{\prime} = y~$ and we define
$~(x,y)\cdot (y,z) = (x,z)~$ and the inverse of $~(x,y)~$ is
defined by $~(x,y)^{-1}:= (y,x).~$ This is usually called the {\it
pair} or {\it coarse groupoid}. Its unit set is $
G_{0}:=\Delta_{X}.~$ The isotropy group $~G(u)~$ at $~u=(x,x)~$ is
the nul group $~\{(u,u)\}.$
\end{exmp}

\begin{exmp}
{\it (i)}~{\it The symmetry groupoid $ {\cal
SG}(X) $}.
Let $ X $ be a nonempty set and consider\\[0.2cm]
 $G:= {\cal SG}(A,X) =\{ f : A \to A~|~ \emptyset\neq A\subseteq X,~ f~ \hbox{is bijective}~ \}
 ~$ and\\[0.2cm]
 $~ G_{0}:= \{ Id_{A}~|~\emptyset \neq A\subseteq X \} $,
 where $ Id_{A} $ is the identity map on $ A.$\\

 Let $ G_{(2)}:= \{(f,g)\in G\times G | D(f)=D(g) \} $, where $ D(f)$ denotes  the domain of  $ f $.
The structure functions $~ \alpha, \beta : G \to G_{0},~ \iota : G
\to G~ $ and the multiplication $~ m : G_{(2)} \to G ~$
  are given by:\\[0.2cm]
$ \alpha(f):= Id_{D(f)},~~ \beta(f):= Id_{D(f)},~~ \iota(f):=
f^{-1}~~ $ and $~~ m(f,g):= f\circ g $.

Then $ (G, G_{0}) $ is a groupoid, called the {\it groupoid of
bijective functions from the subsets $ A $ of $ X $ onto $ A $} or
the {\it symmetry groupoid of the set $X$}.

The isotropy group at $ u= Id_{A}$ is the symmetry group of the set
$A$, i.e. $ G(u) = \{ f : A \to A~|~ f~ \hbox{is bijective}~ \}$.

In particular, the  symmetry groupoid of a finite set $ X = \{
x_{1}, x_{2},\ldots, x_{n} \}, $ is called the {\it symmetry
groupoid of degree $ n $} and is denoted by $ {\cal SG}_{n}$. Its
unit set is $~ {\cal SG}_{n,0} = \{ Id_{A} ~|~ \emptyset \neq A
\subseteq \{ x_{1}, x_{2},\ldots, x_{n}\} \}$. The cardinals of
these finite sets are given by:

$~~~~~~~~~~~~~~~|~{\cal SG}_{n}~|~=~\sum\limits_{k=1}^{n} k!  {n
\choose k},~~~~~ |~{\cal SG}_{n,0}~| ~=~2^{n} - 1.$\\

{\it (ii)~} {\it The Galois groupoid $ {\cal G}al( {\cal E}/ K ) $}.
Let $ F / K $ be an extension field of a field $K$, i.e. $K$ is a
subfield of $F$. We consider an indexed family $ {\cal E}:=
(E_{i})_{i\in I} $ of intermediate fields $ E_{i} $, that is $ K
\subseteq E_{i}
\subseteq F $ for each $ i\in I.$ Let \\[0.2cm]
 $\Gamma:= {\cal G}al({\cal E}/K) = \{ \varphi : E_{i} \to E_{i}~|~ \varphi~ \hbox{is a K -automorphism } \}
 ~$ and\\[0.2cm]
 $~ \Gamma_{0}:= {\cal G}al({\cal E}/K)_{0} = \{ Id_{E_{I}}~|~ i\in I\} $.

 Let $ \Gamma_{(2)}:= \{(\varphi ,\psi)\in \Gamma \times \Gamma | D(\varphi)=D(\psi) \} $.
The structure functions $~ \ov{\alpha}, \ov{\beta} : \Gamma \to
\Gamma_{0},~ \ov{\iota} : \Gamma \to \Gamma~ $ and  $~ \ov{m} :
\Gamma_{(2)} \to \Gamma ~$
  are given by:\\[0.2cm]
$ \ov{\alpha}(\varphi):= Id_{D(\varphi)},~~ \ov{\beta}(\varphi):=
Id_{D(\varphi)},~~ \ov{\iota}(\varphi):=
\varphi^{-1}~~ $ and $~~ \ov{m}(\varphi,\psi):= \varphi \circ \psi $.\\

Then $ {\cal G}al({\cal E}/K) $ is a groupoid over  $ {\cal
G}al({\cal E}/K)_{0}$,  called the {\it Galois groupoid associated
to $ {\cal E}$}. The isotropy group at $ u= Id_{E_{i}}$ is the
Galois group  $ Gal(E_{i}/K)$.
\end{exmp}

\begin{defn}
(\cite{cdw87}) By \textbf{morpfism of groupoids} or {\bf groupoid
morphism} between  the groupoids $(G ,\alpha ,\beta , m ,\iota
,G_0)$ and $(G ^{\prime },\alpha ^{\prime },\beta ^{\prime }, m
^{\prime },\iota ^{\prime} , G_{0}^{\prime})$, we mean a map $~f:G
\rightarrow G ^{\prime } $ which verifies the following
conditions:

{\it(i)} $~~~~~\forall~ (x,y)\in G_{(2)}~~\Longrightarrow~~ (
f(x), f(y) )\in G_{(2)}^{\prime};~$

{\it(ii)}$~~~~~f( m(x,y)) = m^{\prime}(f(x),f(y)), ~ \forall ~
(x,y)\in G_{(2)}.$
\end{defn}

\begin{prop} {\it If $ f : G \longrightarrow G ^{\prime}$ is a morpfism of groupoids, then:
\[\hbox{{\it(a)}}~~~ f\left( u\right)\in G_{0}^{\prime },~~~ \forall~ u\in
G_{0};~~~~~ \hbox{{\it(b)}}~~ f\left( x^{-1}\right) =\left(
f\left( x\right) \right) ^{-1},~  \forall~ x\in G.\]}
\end{prop}

From Proposition 2.3(a) follows that a  groupoid morphism $ f : G
\to G^{\prime} $ induces a map $ f_{0} : G_{0} \to G_{0}^{\prime}
$ taking $ f_{0}(u): = f(u), ~(\forall) u\in G_{0}$, i.e. the map
$ f_{0}$ is the restriction of $ f$ to $ G_{0}.$ We say that
$(f,f_{0}) : (G,G_{0}) \rightarrow ( G^{\prime}, G_{0}^{\prime} )
$ is a morphism of groupoids.

If $~G_{0} = G_{0}^{\prime}~$ and $~f_{0} = Id_{G_{0}},$ we say
that $~f : G \to G^{\prime}~$ is a $ G_{0}~$- {\it morphism} of
groupoids over $ G_{0}$.

A groupoid morphism $~(f,f_{0})~$ is said to be {\it isomorphism
of groupoids} or {\it groupoid isomorphism}, if $ f $ and $ f_{0}$
are bijective maps.

\begin{prop}
(\cite{ivan02}) {\it Let $(G ,\alpha ,\beta, m ,\iota ,G_0)$ and
$(G^{\prime},\alpha^{\prime }, \beta^{\prime}, m^{\prime
},\iota^{\prime },G_{0}^{\prime})$  be two groupoids. The pair
$~(f,f_{0}) : (G, G_{0})\longrightarrow (G^{\prime},
G_{0}^{\prime})~$ where $ f:G \longrightarrow G^{\prime}$ and $ ~
f_{0} : G_{0}\longrightarrow G_{0}^{\prime},$ is a groupoid
morphism if and only if the following conditions are verified:

{\it(i)} $~~~~~\alpha^{\prime} \circ f = f_{0}\circ \alpha \quad
\hbox{and}\quad  \beta^{\prime} \circ f = f_{0}\circ \beta ;$

{\it(ii)}$~~~~~f\left( m \left( x,y\right) \right) = m^{\prime
}\left( f\left( x\right) ,f\left( y\right) \right),\quad  \forall
~ (x,y)\in G_{\left(2\right)}.$}
\end{prop}

\begin{rem}
Applying Propositions 2.3 and 3.4 we can
conclude that a groupoid morphism $~(f,f_{0}) : (G,
G_{0})\longrightarrow (G^{\prime}, G_{0}^{\prime})~$ is linked
with the structure functions by the relations :
\begin{equation}
\alpha^{\prime} \circ f = f_{0} \circ \alpha ,~~ \beta^{\prime}
\circ f = f_{0} \circ \beta ,~~ m^{\prime} \circ (f \times f) =
f\circ m,~~ \iota^{\prime} \circ f = f \circ \iota \label{2.1}
\end{equation}
where $~(f\times f)(x,y):=(f(x),f(y)),~ \forall~ x,y\in G\times
G.$
\end{rem}

\begin{defn}
(\cite{dumoiv}) A groupoid morphism $~( f,
f_{0} ):(G, G_{0})\longrightarrow (G^{\prime},G_{0}^{\prime})~$
satisfying the following condition:
\begin{equation}
\forall~ x,y\in G~~\hbox{such that}~~(f(x),f(y))\in
G_{(2)}^{\prime}~~~\Rightarrow~~~(x,y)\in G_{(2)}\label{2.2}
\end{equation}
will be called {\bf strong morphism} or {\bf homomorphism of
groupoids}.
\end{defn}

\begin{exmp}
Let the  symmetry groupoid $ {\cal SG}_{n}$ of the finite set\\
$ X = \{ x_{1}, x_{2},\ldots, x_{n} \} $ and the multiplicative
group $\{ +1, -1 \}$ ( regarded as groupoid over $\{ +1\}$ ). We
define the map\\[0.2cm]
 $~~~~~~~ sgn^{\sharp}: {\cal SG}_{n} \to \{ +1, -1
\},~ f\in {\cal SG}_{n} \longmapsto sgn^{\sharp}(f):= sgn(f)
$,\\[0.2cm]
where $ sgn(f)$ is the signature of the permutation $f$ of degree
$ k = | D(f)| $.

We have that {\it $~ sgn^{\sharp}: {\cal SG}_{n} \to \{ +1, -1
\}~$ is a groupoid morphism}.

Indeed, let $ f,g \in G_{(2)} $, where $ G = {\cal SG}(A,X)$ such
that $ D(f) = D(g):= A_{k}:=\{ x_{j_{1}},\ldots,
x_{j_{k}}\}\subseteq X, ~ 1\leq k \leq n.$  Then $ f $ and $ g $
are permutations of $ A_{k} $ and $ f\circ g $ is also a
permutation of $ A_{k}$. It is clearly that the condition (i) from
Definition 2.2 is verified. Also, it is well known that\\
 $sgn(f\circ g) = sgn(f)\cdot sgn(g)$. Hence
 $~ sgn^{\sharp}(m(f,g)) =
sgn^{\sharp}(f)\cdot sgn^{\sharp}(g).~$ Therefore the condition
(ii) from Definition 2.2 holds.

The map $~ sgn^{\sharp}: {\cal SG}_{n} \to \{ +1, -1 \}$ {\it is
not a  groupoid homomorphism}.

Indeed, for $ X = \{ x_{1}, x_{2}, x_{3}, x_{4} \} $ we consider
the permutations $ f, g \in {\cal SG}_{4} $, where $ f = \left (
\begin{array}{ccc} x_{1} & x_{2} &
x_{3}\\
x_{2} & x_{3} & x_{1}\\
\end{array}\right ) $ and
 $ g = \left ( \begin{array}{ccc} x_{1} & x_{3} &
x_{4}\\
x_{4} & x_{3} & x_{1}\\
\end{array}\right ).$
Then\\
 $~ sgn^{\sharp}(f) = + 1, ~ sgn^{\sharp}(g) = -1 $ and
 $~( sgn^{\sharp}(f), sgn^{\sharp}(g))\in \{ +1, -1\}\times \{ +1, -1\}
$. But $ f $ and $g$ are not composable in $ {\cal SG}_{4} $,
since $ D(f)\neq D(g).$
\end{exmp}

\section{VECTOR GROUPOIDS}

\begin{defn}
A {\bf vector groupoid over a field $K$}, is a groupoid $~(V,
\alpha, \beta, \odot, \iota,  V_0)$ such that:
\smallskip

\noindent(3.1.1) $~V$ is a vector space over $K$, and the units
set $V_0$ is a subspace of $V$.

\noindent(3.1.2) The source and the target maps $\alpha$ and
$\beta$ are linear maps.
\smallskip

\noindent(3.1.3) The inversion $~\iota:V\longrightarrow V,\
x\longmapsto \iota(x):=x^{-1}$ is a linear map and the
following condition is verified:\\[0.2cm]
$(1)~~~~~~~~~~~ x+x^{-1}=\alpha(x)+\beta(x),\ \mbox{for all } x\in
V.$
\smallskip

 \noindent(3.1.4) The map $~m:V_{(2)}:=\{(x,y)\in V\times V~|~\alpha(y)=\beta(x)\}
\to V,$
$(x,y)\longmapsto m(x,y):=x\odot y,~$  satisfy the
following conditions :
\begin{enumerate}
  \item $x\odot(y+z-\beta(x))=x\odot y+x\odot z-x$, for all $x,y,z\in V$, such that $\alpha(y)=\beta(x)=\alpha(z)$.
  \item $x\odot(ky+(1-k)\beta(x))=k(x\odot y)+(1-k)x$, for all $x,y\in V$, such that $\alpha(y)=\beta(x)$.
  \item $(y+z-\alpha(x))\odot x=y\odot x+z\odot x-x$, for all $x,y,z\in V$, such that $\alpha(x)=\beta(y)=\beta(z)$.
  \item $(ky+(1-k)\alpha(x))\odot x=k(y\odot x)+(1-k)x$ for all $x,y\in V$, such that $\alpha(x)=\beta(y)$.
\end{enumerate}
\end{defn}
When there can be no confusion we put $xy$ or $x\cdot y$ instead
of $x\odot y$.

From Definition 3.1 follows the following corollary.

\begin{cor}
Let $~(V, \alpha, \beta, \odot, \iota,  V_0)$ be a vector
groupoid. Then:

$(i)~~~$ The source and target $ \alpha, \beta : V \to V_{0} $ are
linear epimorphisms.

$(ii)~~$ The inversion $ \iota : V \to V $ is a linear
automorphism.

$(iii)~$ The fibres $\alpha^{-1}(0)$ and $ \beta^{-1}(0) $ and the
isotropy group\\ $ V(0):= \alpha^{-1}(0)\cap \beta^{-1}(0) $ are
vector subspaces of the vector space $V$.
\end{cor}

\begin{exmp}
 Let $V$ be a vector space over a field $K$. If we define the
maps $ \alpha_{0}, \beta_{0}, \iota_{0}: V\longrightarrow V,\ \
\alpha_{0}(x)=\beta_{0}(x)=0,\ \iota_{0}(x)=-x,$ and the
multiplication law $m_{0}(x,y)=x+y$, then $(V, \alpha_{0},
\beta_{0}, m_{0}, \iota_{0}, V_{0}=\{0\})$ is a vector groupoid
called  {\it vector groupoid with a single unit}. We will denote
this vector groupoid by $(V, + )$. Therefore, each vector space $
V $ over $K$ can be regarded as vector groupoid over $V_{0}
=\{0\}$.\hfill$\b$
\end{exmp}

\begin{exmp}
Let $V$ be a vector space over a field $K$. Then $ V$ has a
structure of null groupoid over $V$ ( see Example 2.1(ii) ). In
this case the structure functions are  $~\alpha = \beta =\iota =
Id_{V}$  and $ x \odot x = x $ for all $x\in V$. We have that
$V_{0} = V$ and the maps $ \alpha, \beta, \iota $ are linear.
Since $ x+\iota(x) = x+ x $ and $ \alpha(x) + \beta(x) = x+x $
imply that the condition 3.1.3(1) holds. It is easy to verify the
conditions 3.1.4(1)- 3.1.4(4) from Definition 3.1. Then $V$ is a
vector groupoid, called the {\it null vector groupoid} associated
to $V$.\hfill$\b$
\end{exmp}

\begin{exmp}
 Let $V$ be a vector space over a field $K$. We
consider the pair groupoid $( V\times V, \w{\alpha}, \w{\beta},
\w{m},\w{\iota}, \Delta_{V})$ associated to $V$ ( see Example
2.1(iii)). We have that $ V\times V$ is a vector space over $K$
and the source $\w{\alpha}$ and target $\w{\beta}$ are linear
maps. Also, the inversion $\w{\iota} : V\times V \to V\times V$ is
a linear isomorphism. Therefore it follows that the conditions
$(3.1.1)- (3.1.3) $ are satisfied. By a direct computation we
verify that the relations 3.1.4(1) - 3.1.4(4) from Definition 3.1
hold. Hence $V\times V$ is a vector groupoid called the {\it
coarse vector groupoid} or {\it pair vector groupoid} associated
to $V$.\hfill$\b$
\end{exmp}

\begin{exmp}
{\bf The vector groupoid $V^2(p,q)$}. Let $V$ be a vector space
over a field $K$ and let $p,q\in K$ such that $pq=1$. The maps
$\alpha,\beta,\iota:V^2\longrightarrow V^2$,
$\alpha(x,y):=(x,px)$, $\beta(x,y):=(qy,y)$, $\iota(x,y):=(qy,px)$
together with the multiplication law given on
$V^2_{(2)}:=\{((x,y),(qy,z))\ |\ x,y,z\in V\}\subset V^2\times
V^2$, by $(x,y)\cdot(qy,z):=(x,z)$ determine on $V^2$ a structure
of vector groupoid. This is called the {\it pair} or the {\it
coarse vector groupoid of type $(p,q)$} and it is denoted by
$V^2(p,q)$.\\
\indent If $p=q=1$, then the vector groupoid $V^2(1,1)$ coincide
with the pair vector groupoid associated to $V$(see
Example 3.3 ).\\
\indent If $n$ is a prime number and $p,q\in \mathbb{Z}_n$, such
that $pq=1$, then $\mathbb{Z}_n^2(p,q)$ is called the
\textit{modular} or \textit{cryptographic vector
groupoid}.\hfill$\b$
\end{exmp}

\begin{exmp}
Let $V$ be vector space over a field $K$. One consider the maps
 $\alpha, \beta, \iota:V^3\longrightarrow V^3$, $\alpha(x_1,x_2,x_3):=(x_1,x_1,0)$, $\beta(x_1,x_2,x_3):=(x_2,x_2,0)$,
 $\iota(x_1,x_2,x_3):=(x_2,x_1,-x_3)$ together with the multiplication law given on $V_{(2)}^3=\{((x_1,x_2,x_3),(x_2,y_2,y_3))~|~ x_1,x_2,x_3,y_2,y_3\in V\}\subset V^3 \times V^3$ by
  $(x_1,x_2,x_3)\odot(x_2,y_2,y_3):=(x_1,y_2,x_3+y_3)$.

  Then $(V^3,\alpha,\beta,\iota,\odot,V_0^3)$, where $V_0^3=\{(x,x,0)\ |\ x\in V\}$, is a vector groupoid.\hfill$\b$
\end{exmp}

\indent In the following proposition, we give, in addition to
those in Proposition 2.1, other rules of algebraic calculation in
a vector groupoid.
\renewcommand{\labelenumi}{(\roman{enumi})}
\begin{prop}
In a vector groupoid $(V, \alpha, \beta, \odot, \iota, V_0)$ the
following assertions hold :
\begin{enumerate}
  \item $0\cdot x=x,\ \forall\ x\in \alpha^{-1}(0)$;
  \item $x\cdot 0=x,\ \forall\ x\in \beta^{-1}(0)$;
  \item For all $x,\ y\in\beta^{-1}(0)$, we have $x-\alpha(x)=y-\alpha(y)\Longrightarrow x=y$;
  \item for all $x,y\in\alpha^{-1}(0)$, we have $x-\beta(x)=y-\beta(y)\Longrightarrow x=y$.
\end{enumerate}
\end{prop}
\begin{proof}
(i) If $x\in\alpha^{-1}(0)$, then $\alpha(x)=0=\beta(0)$. So $(0,x)\in V_{(2)}$ and, using the condition
{\it (G2)} from Definition 2.1, one obtains that $0\cdot x=\alpha(x)\cdot x=x$.\\
(iv) Let $x,y\in\alpha^{-1}(0)$ such that $x-\beta(x)=y-\beta(y)$.
Then $ \alpha(x)= \alpha(y) = 0 $ and $ x - y = \beta(x) -
\beta(y).$  Since $ \alpha $  is linear map and applying
Proposition 2.1 (i), one obtains that $0=\alpha(x)- \alpha(y) =
\alpha(x-y)=\alpha (\beta(x)- \beta(y))=\beta(x)-\beta(y)= x-y $,
and so $x=y$.

Similarly, we prove that the assertions (ii) and (iii) hold.
\end{proof}

\begin{prop}
Let $(V, \alpha, \beta, \odot, \iota, V_0) $ be a vector groupoid.
Then:

$(i)~~~t_{\beta}:\alpha^{-1}(0)\longrightarrow\beta^{-1}(0),~
t_{\beta}(x):=\beta(x)-x~$ is a linear isomorphism.

$(ii)~~t_{\alpha}:\beta^{-1}(0)\longrightarrow\alpha^{-1}(0),~
t_{\alpha}(x):=\alpha(x)-x~$ is a linear isomorphism.
\end{prop}
\begin{proof} $(i)$ Let be $ x_{1}, x_{2} \in V $ and $ k_{1},
k_{2} \in K$. Then $ t_{\beta}( k_{1} x_{1} + k_{2} x_{2})=$\\
$=\beta( k_{1} x_{1} + k_{2} x_{2})- ( k_{1} x_{1} + k_{2} x_{2})=
k_{1}( \beta(x_{1}) - x_{1}) + k_{2}( \beta(x_{2}) - x_{2}) =$\\
$=k_{1}t_{\beta}(x_{1}) + k_{2}t_{\beta}(x_{2}). $ Hence $
t_{\beta} $ is a linear map.

Let now $ x,y \in \alpha^{-1}(0) $ such that $ t_{\beta} (x) =
t_{\beta} (y)$. Applying Proposition 3.1(iv), one obtains $ x=y$,
and so the map $ t_{\beta} $ is injective.

For any $ y\in \beta^{-1}(0)$ we take $ x = \alpha(y)-y $. Clearly
$ x\in \alpha^{-1}(0).$ We have $ t_{\beta}(x) =
\beta(\alpha(y)-y) - (\alpha(y)-y ) = \alpha(y)- \beta(y) -
\alpha(y) + y = y $, since $ \beta(y) = 0$. Hence the map $
t_{\beta} $ is surjective. Therefore $ t_{\beta} $ is a linear
isomorphism.

$(ii)~$ Similarly we prove that $ t_{\alpha}$ is a linear
isomorphism.
\end{proof}

\begin{prop}
Let $~(V, +, \cdot, \alpha, \beta, \odot, \iota, V_0)~$ be a
vector groupoid over $ K $ and $~ u\in V_{0}~$ any unit of $~V$.
The following assertions hold.

$(i)~~~$ The isotropy group $ V(u):=\{x\in V~|~
\alpha(x)=\beta(x)=u \}$ endowed with the laws $~\boxplus :
V\times V \to V~$ and $~\boxtimes : K\times V \to V~$ given by:
\begin{equation}
x\boxplus y=x+y-u,~~~\forall~x,y\in V(u)\label{3.1}
\end{equation}
\begin{equation}
k\boxtimes x = k x + (1-k)u,~~~\forall~k \in K,~x\in
V(u),\label{3.2}
\end{equation}
has a structure of vector space over $K$.

$(ii)~~$ The vector space $~( V(u), \boxplus, \boxtimes )~$
together with the restrictions of structure functions $~\alpha,
\beta, \iota~$ to $~V(u)~$ and the multiplication\\
$~\boxdot:V(u)_{(2)}= V(u)\times V(u) \to V(u)~ $ given by:
\begin{equation}
x\boxdot y = (x-u)\odot(y-u) + u ,~~~\forall~x,y\in
V(u)\label{3.3}
\end{equation}
has a structure of vector groupoid with a single unit over $K$.
\end{prop}
\begin{proof}
$(i)~$ Using the linearity of the functions $\alpha$ and $\beta$
we verify that the laws $\boxplus$ and $\boxtimes$ given by (3.1)
and (3.2) are well-defined. For instance, for $x,y\in V(u)$ we
have $ \alpha(x\boxplus y) = \alpha(x+y-u)=\alpha(x)+ \alpha(y) -
\alpha(u) = u,$ since $ \alpha(x)=\alpha(y)=\alpha(u)=u.$
Similarly, $\beta( x\boxplus y ) =u.$ Hence $ x\boxplus y \in
V(u).$
 It is easy to verify that $ (V(u), \boxplus ) $ is a commutative
 group. Its null vector is the element $ u\in V(u)$. The opposite
 $~\boxminus x~$ of $~x\in V(u)~$ is $~\boxminus x = 2 u -
 x$.

 For any $ x, y\in V(u) $ and  $ k, k_{1}, k_{2} \in K $, the law $\boxtimes $ verify the following relations:

$(a)~~~k\boxtimes ( x\boxplus y ) = ( k\boxtimes  x ) \boxplus (k
\boxplus y ),$\\

$(b)~~~(k_{1} + k_{2})\boxtimes  x = (k_{1} \boxtimes x ) \boxplus
( k_{2}\boxtimes x ),$\\

$(c)~~~k_{1}\boxtimes (  k_{2}\boxtimes  x ) = (k_{1}k_{2})\boxtimes x ,$\\

$(d)~~~ 1 \boxtimes  x  = x $ ( here $1$ is the unit of the field
$K$ ).\\

Indeed, we have $~k\boxtimes ( x\boxplus y ) = k ( x\boxplus y ) +
(1 -k )u = k ( x + y ) + (1 -2 k )u ~$ and $~( k\boxtimes  x )
\boxplus (k \boxplus y )= ( k\boxtimes  x ) + (k \boxplus y ) - u
= k ( x +  y ) + (1 - 2k )u ~$.  Hence the equality (a) holds.

 In the same manner we prove that the equalities (b) - (d) hold.
Therefore $( V, \boxplus, \boxtimes ) $ is a vector space.

$(ii)~$ From the above assertion follows that the condition
(3.1.1) from Definition 3.1 is satisfied.

The restrictions of the linear maps $ \alpha $ and $\beta $ to
$V(u)$ are linear maps, and so the condition (3.1.2) from
Definition 3.1 holds.

Also, the restriction of the linear maps $ \iota $  to $V(u)$ is
linear map. Applying the equality 3.1.3(1) from Definition 3.1,
for any $ x\in V(u)$ we have\\
 $ x \boxplus \iota(x) = x +
\iota(x) - u = \alpha(x) + \beta(x)- u = \alpha(x) \boxplus
\beta(x).~$ Therefore the condition (3.1.3) from Definition 3.1
holds.

Let $ x,y \in V(u)$. Applying the properties of maps $ \alpha $
and $ \beta $ we have $ \alpha (x\boxdot y ) = \alpha (
(x-u)\odot(y-u) + u ) = \alpha ( (x-u)\odot(y-u)) + \alpha (u)
=$\\
$=\alpha (x-u) + \alpha(u) = \alpha(x) = u $ and $ \beta (x\boxdot
y )= u $ and so $ x\boxdot y \in V(u).$ Hence the law $ \boxdot$
given by the relation (3.3) is well-defined.

If $ x, y, z \in V(u)$ then the following equality holds:\\

$(e)~~~ x\boxdot ( y \boxplus z \boxplus ( \boxminus \beta(x))) =
( x\boxdot y ) \boxplus (  x\boxdot z  ) \boxplus ( \boxminus
x)).$\\

Indeed, we have\\[0.2cm]
$(e.1)~~ x\boxdot ( y \boxplus z \boxplus ( \boxminus \beta(x))) =
x\boxdot ( y \boxplus z \boxplus ( \boxminus u)) = x\boxdot ( y
\boxplus z \boxplus u )=$\\
$= x\boxdot ( y \boxplus z )= (x - u)\odot ( y \boxplus z - u ) +
u = (x - u)\odot ( (y - u) + (z - u )) + u. $

Replacing  in the equality 3.4.1(1) the elements $ x, y, z \in
V(u)~$  respectively with $~ x-u, y-u, z-u \in V(u),~$ we obtain
the following equality\\[0.2cm]
$(f)~~ (x - u)\odot (( y - u) +  (z - u ))= (x - u)\odot  (y - u)
+ (x - u )\odot (z-u) - (x - u), $\\[0.2cm]
since $ \beta(x-u) = 0.$

Using the relation (f), the equality $(e.1)$ becomes\\[0.2cm]
$(e.2)~~ x\boxdot ( y \boxplus z \boxplus ( \boxminus \beta(x))) =
(x-u)\odot (y-u) + (x-u)\odot (z-u)+ 2 u - x.$\\

On the other hand we have\\[0.2cm]
$(e.3)~~~ ( x\boxdot y ) \boxplus (  x\boxdot z  ) \boxplus (
\boxminus x))=( ( x\boxdot y ) \boxplus (  x\boxdot z  )) \boxplus
( 2 u - x)=$\\
$= (  x\boxdot y  +  x\boxdot z - u  ) \boxplus ( 2 u - x) =
x\boxdot y  +  x\boxdot z -  x =$\\
$= (x-u)\odot (y-u) + (x-u)\odot (z-u)+ 2 u - x.$\\

Using (e.2) and (e.3) we obtain the equality (e). Hence, the
relation 3.4.1(1) from Definition 3.1 holds.

In the same manner we can prove that the relations 3.1.4(2) -
3.1.4(4) from Definition 3.1 are verified.
\end{proof}

\indent We call $(V(u), \boxplus, \boxtimes, \alpha, \beta,
\boxdot, \iota, V_{0}(u)=\{u\}) $ the {\it isotropy vector
groupoid} at $u\in V_{0}$ of $V$, when one refers to the above
structure given on it.

\begin{defn}
Let $ ( V_1, \alpha_1, \beta_1, V_{1,0} )$ and  $ ( V_2, \alpha_2,
\beta_2, V_{2,0} )$ be two vector groupoids.

A groupoid morphism ( resp. groupoid homomorphism ) $
f:V_1\longrightarrow V_2 $ with property that $ f $ is a linear
map, is called  {\bf vector groupoid morphism} ( resp. {\bf vector
groupoid homomorphism} ).
\end{defn}

\begin{exmp}
Let $(V,\alpha, \beta, \odot , \iota , V_{0})$ be a vector
groupoid. We consider the pair vector groupoid $ ( V_{0}\times
V_{0}, \w{\alpha}, \w{\beta}, \w{m}, \w{\iota}, \Delta_{V_{0}}) $.
Then\\
 {\it the anchor map $(\alpha, \beta): V \to V_{0}\times
V_{0} $ is a homomorphism of vector groupoids between the vector
groupoids $ V $ and $ V_{0}\times V_{0}$}.

Indeed, if we denote $ (\alpha, \beta):=f $ and consider the
elements $ x,y\in G $ such that  $ ( f(x), f(y))\in (V_{0}\times
V_{0})_{(2)}$, then $ \w{\beta}(f(x))= \w{\alpha}(f(y)) $ and we
have $ \w{\beta}(\alpha(x), \beta(x))= \w{\alpha}(\alpha(y),
\beta(y))~\Rightarrow ~ (\beta(x),
\beta(x))=(\alpha(y),\alpha(y))~\Rightarrow~ \beta(x))=\alpha(y)
$, i.e. $ (x,y)\in V_{(2)}. $ Therefore the condition (i) from
Definition 2.2 holds.

For $(x,y)\in V_{(2)}$ we have

 $ f(m(x,y))= f(xy)=(\alpha(xy),
\beta(xy))= (\alpha(x), \beta(y)) ~$ and

 $ \w{m}(f(x),f(y))=
\w{m}( (\alpha(x),\beta(x)), (\alpha(y),\beta(y))) =
(\alpha(x),\beta(y)).$

\indent Hence the equality (ii) from Definition 2.2 is verified.

Let now two elements $ x,y \in V$ such that $(f(x),f(y))\in
(V_{0}\times V_{0})_{(2)}.$ Then $ \w{\beta}(f(x)) =
\w{\alpha}(f(y))$. Since $ f(x) = (\alpha(x), \beta(x)) $ and $
f(y) = (\alpha(y), \beta(y)) $ we deduce that $ (\beta(x),
\beta(x))=(\alpha(y), \alpha(y)) $. Therefore $ \beta(x) = \alpha
(y) $ and $(x,y)\in G_{(V)}.$ Therefore the condition (2.2) from
Definition 2.3 is satisfied.

Hence $ f : V \to V_{0}\times V_{0} $ is a groupoid homomorphism.

Let $ x, y \in V $ and $ a,b\in K$. Since $ \alpha, \beta $ are linear maps, we have\\
$ f(ax+by) = (\alpha(ax+by), \beta(ax+by)) = ( a \alpha(x) +
b\alpha(y), a \beta(x) + b\beta(y))= $\\
$= a(\alpha(x) , \beta(x)) + b (\alpha(y) , \beta(y))= a f(x) + b
f(y) $, i.e. $f$ is a linear map.

Therefore, the conditions from Definition 3.2 are verified. Hence
$f$ is a vector groupoid homomorphism.\hfill$\b$
\end{exmp}

\section{ALGEBRAIC CONSTRUCTIONS OF VECTOR GROUPOIDS}

In this section we shall give some important ways of building up
new vector groupoids.

{\bf 1. Direct product of two vector groupoids}. Let given the
vector groupoids $ ( V, \alpha_{V}, \beta_{V},
\odot_{V},\iota_{V}, V_{0}) $ and $ ( W, \alpha_{W}, \beta_{W},
\odot_{W}, \iota_{W}, W_{0}) $. We have that $ V_{0}\times W_{0}$
is a vector subspace of the direct product $ V\times W$  of vector
spaces $ V$ and $W$.

We can easy prove that $ V\times W$ endowed with the structure
functions $ \alpha_{V\times W}, \beta_{V\times
W}, \odot_{V\times W} $ and $ \iota_{V\times W} $ given by\\
$ \alpha_{V\times W}(v,w):=( \alpha_{V}(v), \alpha_{W}(w)),~
\beta_{V\times W}(v,w):=( \beta_{V}(v), \beta_{W}(w)),$\\
$(v_{1}, w_{1})\odot_{V\times W} (v_{2}, w_{2}):= ( v_{1}
\odot_{V} v_{2}, w_{1} \odot_{W} w_{2} ), ~ \iota_{V\times W}(v,w):=( \iota_{V}(v), \iota_{W}(w))$ \\
 for all $ v, v_{1}, v_{2} \in V $ and $ w, w_{1}, w_{2} \in W $,
is a vector groupoid over $V_{0}\times W_{0}$.

This vector groupoid is called the {\it direct product of vector
groupoids} $(V,V_{0})$ and $(W,W_{0})$.

By a direct computation we can verify that the projections\\
 $pr_{V} : V\times W \to V$ and $ pr_{W} : V\times W \to W$ are
morphisms of vector groupoids, called the {\it canonical
projections} of the vector groupoid $ V\times W$ onto vector
groupoid $V$ and $ W$, respectively. The following assertion holds

{\it  The direct product of two transitive vector groupoids is
also a transitive vector groupoid}.

{\bf 2. Trivial vector groupoid  ${\cal TVG}(V,W) $}. Let $W$ be a
vector subspace of a vector space $V$ over $K.$ The set $~{\cal
V}:= W\times V \times W ~$ has a natural structure of vector
space. The set $ {\cal V}_{0}: = \{ (w, 0, w)\in {\cal V} ~|~ w\in
W\} $ is a vector subspace of ${\cal V}$ (here $ 0$ is the null
vector of $V$ ). We introduce on $ {\cal V}:= W\times V \times W $
the structure functions $ \alpha_{\cal V}, \beta_{\cal V},
\odot_{\cal V} $ and $ \iota_{\cal V}$ as follows.

For all $~(w_{1},v,w_{2})\in {\cal V} $, the source and target $
\alpha_{\cal V}, \beta_{\cal V}: {\cal V} \to {\cal V}_{0} $ are
defined by
\smallskip

$\alpha_{\cal V}(w_{1},v, w_{2}):= (w_{1}, 0, w_{1});~~~
\beta_{\cal V}(w_{1},v, w_{2}):= (w_{2}, 0, w_{2}).$
\smallskip

 The partially multiplication  $\odot_{\cal V}: {\cal V}_{(2)} \to {\cal V} $,
where

 $ {\cal V}_{(2)} = \{ ( (w_{1},
v_{1}, w_{2}), (w_{2}^{\prime}, v_{2}, w_{3}) ) \in {\cal V}\times
{\cal V} ~|~ w_{2} = w_{2}^{\prime} \}$  and the inversion map $
\iota_{\cal V}: {\cal V}\to {\cal V}$ are given by
\smallskip

$ (w_{1}, v_{1}, w_{2})\odot_{\cal V} (w_{2}, v_{2}, w_{3}):= (
w_{1}, v_{1}+ v_{2}, w_{3});~~ \iota_{\cal V}(w_{1}, v, w_{2}):=(
w_{2}, -v, w_{1} )$.\\

 It is easy to verify that the conditions of Definition
2.1 are satisfied. Then $~( {\cal V}, \alpha_{\cal V}, \beta_{\cal
V}, \odot_{\cal V}, \iota_{\cal V}, {\cal V}_{0})~$ is
 a groupoid. Also, the condition (3.1.1) from Definition 3.1 is
 verified.

 Let now two elements $ x,y\in {\cal V}$ and $ a,b\in K$
 where $ x = (w_{1}, v_{1}, w_{2}) $ and $ y = (w_{3}, v_{2}, w_{4})
 $. We have

$ \alpha_{\cal V}(a x + b y)= \alpha_{\cal V}(a w_{1}+ b w_{3}, a
v_{1}+ b v_{2}, a w_{2}+ b w_{4})=$\\
$= ( a w_{1}+ b w_{3}, 0, a w_{1}+ b w_{3}) = a( w_{1}, 0,  w_{1})
+ b( w_{3}, 0, w_{3}) = a \alpha_{\cal V}(w_{1}, v_{1}, w_{2})
+$\\
$+b \alpha_{\cal V}(w_{3}, v_{2}, w_{4}) = a \alpha_{\cal V}(x) +
b \alpha_{\cal V}(y).$

 It follows that $ \alpha_{\cal V}$ is a linear map. Similarly we prove that $ \beta_{\cal V}$ is a linear map.
Therefore the conditions (3.1.2) from Definition 3.1 hold.

For $ x = (w_{1}, v_{1}, w_{2})\in {\cal V} $ and $ y = (w_{3},
v_{2}, w_{4}) \in {\cal V}$ and $ a,b\in K$, we have

$ \iota_{\cal V}(a x + b y)= \iota_{\cal V}( a w_{1}+ b w_{3}, a
v_{1}+ b v_{2}, a w_{2}+ b w_{4} )= $\\
$=( a w_{2}+ b w_{4}, - a v_{1} - b v_{2}, a w_{1}+ b w_{3}) = a (
w_{2}, - v_{1} ,  w_{1}) + b ( w_{4}, - v_{2}, w_{3} )=$\\
$= a \iota_{\cal V}( w_{1},  v_{1}, w_{2}) + b \iota_{\cal V}(
w_{3}, v_{2}, w_{4}) = a \iota_{\cal V}(x) + b \iota_{\cal V}(y).
$

It follows that $ \iota_{\cal V}$ is a linear map. Also

$x + \iota_{V}(x) = (w_{1}, v_{1}, w_{2}) + (w_{2}, - v_{1},
w_{1}) = (w_{1} + w_{2}, 0, w_{1} + w_{2}) =$\\
$=(w_{1}, 0, w_{1}) + ( w_{2}, 0, w_{2})= \alpha_{V}(x) +
\beta_{V}(x).$

Hence the condition (3.1.3) from Definition 3.1 holds.

For to verify the relation $ 3.1.4 (1)$ from Definition 3.1 we
consider the arbitrary elements  $ x, y, z \in {\cal V} $ where $
x = (w_{1}, v_{1}, w_{2} ), y = ( w_{3}, v_{2}, w_{4}) $ and $ z =
(w_{5}, v_{3}, w_{6}) $  such that $ \alpha_{\cal V}(y) =
\beta_{\cal V}(x) =\alpha_{\cal V}(z) $.  Then $ w_{2} = w_{3} =
w_{5} $ and follows $ x = (w_{1}, v_{1}, w_{2} ), y = ( w_{2},
v_{2}, w_{4}) $ and $ z = (w_{2}, v_{3}, w_{6}) $.

 For all $ k\in K$ we have

 $(i)~~~ x\odot_{\cal V}( y + z - \beta_{\cal V}(x)) = (w_{1}, v_{1}, w_{2} )\odot_{\cal
 V}( ( w_{2}, v_{2}, w_{4}) +$\\
 $+ (w_{2}, v_{3}, w_{6}) - (w_{2}, 0, w_{2}) )
 =(w_{1}, v_{1}, w_{2} )\odot_{\cal
 V} ( w_{2}, v_{2} + v_{3}, w_{4} + w_{6} - w_{2}) =$\\
 $= ( w_{1}, v_{1} + v_{2} +
 v_{3}, w_{4} + w_{6} - w_{2})~$ and

$(ii)~~~ x\odot_{\cal V} y + x\odot_{\cal V} z - x = (w_{1},
v_{1}, w_{2} )\odot_{\cal V}(  w_{2}, v_{2}, w_{4}) + $\\
$+ (w_{1}, v_{1}, w_{2} )\odot_{\cal V}(  w_{2}, v_{3}, w_{6}) -
(w_{1}, v_{1}, w_{2} )= ( w_{1}, v_{1}+ v_{2}, w_{4} ) + $\\
$+ ( w_{1}, v_{1}+ v_{3}, w_{6} ) -(w_{1}, v_{1}, w_{2} ) = (
w_{1}, v_{1}+ v_{2} + v_{3}, w_{4} + w_{6}- w_{2} ).$

Using (i) and (ii) we obtain $ x\odot_{\cal V}( y + z -
\beta_{\cal V}(x)) = x\odot_{\cal V} y + x\odot_{\cal V} z - x $.
Hence the condition 3.1.4 (1) from Definition 3.1 holds.

Let now $ x = (w_{1}, v_{1}, w_{2} ), y = ( w_{2}, v_{2}, w_{4}) $
and $ k\in K$. We have

$(iii)~~ x \odot_{\cal V}( k y + (1-k)\beta_{\cal V}(x) ) =
(w_{1}, v_{1}, w_{2} )\odot_{\cal V} ( k (w_{2}, v_{2}, w_{4})
+$\\
$+(1-k)(w_{2}, 0, w_{2}) )= (w_{1}, v_{1}, w_{2} )\odot_{\cal V} (
w_{2}, k v_{2}, k w_{4} + (1-k)w_{2}) =$\\
$= ( w_{1}, v_{1}+ k v_{2}, k w_{4} + (1-k)w_{2})~ $ and

$(iv)~~ k ( x \odot_{\cal V} y ) + (1-k) x = k ((w_{1}, v_{1},
w_{2} )\odot_{\cal V} ( (w_{2}, v_{2}, w_{4}) ) +$\\
$+ (1-k)(w_{1}, v_{1}, w_{2} )= k (w_{1}, v_{1} + v_{2}, w_{4} )+
(1-k)(w_{1}, v_{1}, w_{2} ) =$\\
$= ( w_{1}, v_{1}+ k v_{2}, k w_{4} + (1-k)w_{2})~ $

Using the equalities (iii) and (iv) we obtain that the condition
3.1.4 (2) from Definition 3.1 holds.

In the same manner we prove that the conditions 3.1.4 (3) and
3.1.4 (4) hold. Hence $ {\cal V}: = W\times V\times W $ is a
vector groupoid over $ {\cal V}_{0}$. Its set of units  can be
identified with the vector subspace $ W$ of $V$.

The vector groupoid $~( {\cal V}:= W\times V\times W, \alpha_{\cal
V}, \beta_{\cal V}, \odot_{\cal V}, \iota_{\cal V}, {\cal
V}_{0})~$ is called the {\it trivial vector groupoid} associated
to pair of vector spaces $(V, W)$  with $ W \subseteq V. $  This
vector groupoid is denoted by $ {\cal TVG}(V,W)$. The isotropy
group at $~u = ( w, 0 ,w) \in {\cal V}_{0}$ is $~V(u) = \{
(w,v,w)~|~ v\in V\}~$ which identify with  the group $~( V, + )$.

{\bf 3. Whitney sum of two vector groupoids over the same base}.
Let $ ( V, \alpha_{V}, \beta_{V}, \odot_{V},\iota_{V}, V_{0}) $
and  $ ( V^{\prime}, \alpha_{V^{\prime}}, \beta_{V^{\prime}},
\odot_{V^{\prime}},\iota_{V^{\prime}}, V_{0}) $ be two vector
groupoids over the same base ( i.e. $ V$ and $V^{\prime}$ have the
same units). The set\\
$ V\oplus V^{\prime}:= \{~(v,v^{\prime})\in V\times V^{\prime}~|~
\alpha_{V}(v) = \alpha_{V^{\prime}}(v^{\prime}), \beta_{V}(v) =
\beta_{V^{\prime}}(v^{\prime})~\}~$
 has a natural structure of vector space. It is clearly that\\
  $ \Delta_{V_{0}} = \{ (u,u)\in
V_{0}\times V_{0}~|~ u\in V_{0}\} \subseteq V\oplus V^{\prime}~$
 is a vector subspace.

We introduce on  $ {\cal W}:= V\oplus V^{\prime}$ the structure
functions $ \alpha_{\cal W}, \beta_{\cal W}, \odot_{\cal W} $ and
$ \iota_{\cal W}$ as follows.

The source and target $ \alpha_{\cal W}, \beta_{\cal W}: {\cal W}
\to \Delta_{V_{0}}$ are defined by\\

$\alpha_{\cal W}(v,v^{\prime}):= ( \alpha_{V}(v) ,
\alpha_{V}(v));~~~ \beta_{\cal W}(v,v^{\prime}):= ( \beta_{V}(v) ,
\beta_{V}(v)), ~~ (v,v^{\prime})\in {\cal W}.$\\

The partially multiplication  $\odot_{\cal W}: {\cal W}_{(2)} \to
{\cal W} ,$ where\\
 $ {\cal W}_{(2)} = \{ ( (v_{1},
v_{1}^{\prime}), ((v_{2}, v_{2}^{\prime}) ) \in {\cal W}\times
{\cal W} ~|~ \beta_{V}(v_{2}) = \alpha_{V}(v_{1}) \}$ and the
inversion map $ \iota_{\cal W}: {\cal V}\to {\cal W}$ are given
by\\

$ (v_{1}, v_{1}^{\prime})\odot_{\cal W}  (v_{2}, v_{2}^{\prime}):=
( v_{1}\odot_{V} v_{2}, v_{1}^{\prime}\odot_{V^{\prime}}
v_{2}^{\prime} );~~~ \iota_{\cal W}(v,v^{\prime}):= ( \iota_{V}(v)
, \iota_{V^{\prime}}(v^{\prime}))$.\\[0.2cm]

By a direct computation we prove that the conditions of Definition
2.1 are satisfied. Then $~( {\cal W}:= V\oplus V^{\prime},
\alpha_{\cal W}, \beta_{\cal W}, \odot_{\cal W}, \iota_{\cal W},
\Delta_{V_{0}})~$ is
 a groupoid. Also, the condition (3.1.1) from Definition 3.1 is
 verified.

 Let now two elements $ x,y\in {\cal W}$ and $ a,b\in K$
 where $ x = (v_{1}, v_{1}^{\prime}) $ and $ y = (v_{2}, v_{2}^{\prime})
 $.
We have

$ \alpha_{\cal W}(a x + b y)= \alpha_{\cal W}(a v_{1}+ b v_{2}, a
v_{1}^{\prime}+ b v_{2}^{\prime} )= ( \alpha_{V}(a v_{1}+ b
v_{2}), \alpha_{V}(a v_{1}+ b v_{2})) = ( a \alpha_{V}(v_{1})+ b
\alpha_{V}(v_{2}), a \alpha_{V}(v_{1})+ b \alpha_{V}(v_{2})) $ and

$a \alpha_{\cal W}(x) + b \alpha_{\cal W}(y)= a \alpha_{\cal W}(
v_{1}, v_{1}^{\prime}) + b \alpha_{\cal W}( v_{2}, v_{2}^{\prime})
= a ( \alpha_{V}(v_{1}), \alpha_{V}(v_{1}) ) + b (
\alpha_{V}(v_{2}) , \alpha_{V}(v_{2}) )= ( a \alpha_{V}(v_{1}) + b
\alpha_{V}(v_{2}) , a \alpha_{V}(v_{1}) + b \alpha_{V}(v_{2})) $
since $ \alpha_{V}$ is a linear map. It follows that $
\alpha_{\cal W}$ is a linear map.

Similarly we obtain that $ \beta_{\cal W}$ is a linear map.
Therefore the conditions (3.1.2) from Definition 3.1 hold.

For $ x = (v_{1}, v_{1}^{\prime})\in {\cal W} $ and $ y = (v_{2},
v_{2}^{\prime}) \in {\cal W}$ and $ a,b\in K$, we have
successively

$ \iota_{\cal W}(a x + b y)= \iota_{\cal W}(a v_{1}+ b v_{2}, a
v_{1}^{\prime}+ b v_{2}^{\prime} )= ( \iota_{V}(a v_{1}+ b v_{2}),
\iota_{V^{\prime}}(a v_{1}^{\prime}+ b v_{2}^{\prime})) = ( a
 \iota_{V}(v_{1})+ b  \iota_{V}(v_{2}), a
\iota_{V^{\prime}}(v_{1}^{\prime})+ b
\iota_{V^{\prime}}(v_{2}^{\prime})) = a (
 \iota_{V} (v_{1}),\iota_{V^{\prime}}(v_{1}^{\prime}))+ b (\iota_{V} (v_{2}),\iota_{V^{\prime}}(v_{2}^{\prime}))
= a \iota_{W}(v_{1}, v_{1}^{\prime}) + b  \iota_{W} ( v_{2},
v_{2}^{\prime}) = a \iota_{W} (x) + b \iota_{W}(y),$ since
$\iota_{V} $ and $ \iota_{V^{\prime}}$  are linear map.

Using the equalities 3.1.3(1) for the inversion maps $ \iota_{V} $
and $ \iota_{V^{\prime}}$ we have\\
$ x + \iota_{W}(x) = ( v, v^{\prime}) + (
\iota_{V}(v),\iota_{V^{\prime}} (v^{\prime}))=  ( v+ \iota_{V}(v),
 v^{\prime}+ \iota_{V^{\prime}}(v^{\prime}))=$\\
 $= ( \alpha_{V}(v) +
\beta_{V}(v), \alpha_{{V}^{\prime}}(v^{\prime}) +
\beta_{V^{\prime}}(v^{\prime}))= ( \alpha_{V}(v) + \beta_{V}(v),
\alpha_{V}(v) + \beta_{V}(v))=$\\
$= \alpha_{W}(v,v^{\prime}) + \beta_{W}(v,v^{\prime})=
\alpha_{W}(x) + \beta_{W}(x) $ for any $ x =(v,v^{\prime})\in W.$

 Hence the conditions (3.1.3) from Definition 3.1 hold.

For to verify the relation $ 3.1.4 (1)$ from Definition 3.1 we
consider the arbitrary elements  $ x, y, z \in {\cal W} $ where $
x = (v_{1}, v_{1}^{\prime}), y = (v_{2}, v_{2}^{\prime}) $ and $z
= (v_{3}, v_{3}^{\prime}) $. We assume that $\alpha_{\cal W}(y) =
\beta_{\cal W}(x) =\alpha_{\cal W}(z) $.

Applying the properties of the structure functions of the vector
groupoids $ V$ and $ V^{\prime} $, we have

$y + z - \beta_{\cal W}(x) =  (v_{2}, v_{2}^{\prime})+ (v_{3},
v_{3}^{\prime})- \beta_{\cal W}(v_{1}, v_{1}^{\prime})=$\\
$= (v_{2}+ v_{3}, v_{2}^{\prime}+ v_{3}^{\prime})-
(\beta_{V}(v_{1}), \beta_{V}(v_{1}) )= (v_{2}+ v_{3}-
\beta_{V}(v_{1}) , v_{2}^{\prime} + v_{3}^{\prime}-
\beta_{V}(v_{1}) ) =$\\
$= ( v_{2}+ v_{3}- \beta_{V}(v_{1}) , v_{2}^{\prime} +
v_{3}^{\prime}- \beta_{V^{\prime}}(v_{1}^{\prime}) )$ and

$(a)~~ x\odot_{\cal W}( y + z - \beta_{\cal W}(x)) = (v_{1},
v_{1}^{\prime})\odot_{\cal W}(v_{2}+ v_{3}- \beta_{V}(v_{1}) ,
v_{2}^{\prime} + v_{3}^{\prime}-
\beta_{V^{\prime}}(v_{1}^{\prime}) )=$\\
$= (  v_{1} \odot_{V}(v_{2}+ v_{3}- \beta_{V}(v_{1}),
v_{1}^{\prime} \odot_{V^{\prime}}( v_{2}^{\prime} +
v_{3}^{\prime}- \beta_{V^{\prime}}(v_{1}^{\prime}) ).$

On the other hand we have

$(b)~~ x\odot_{\cal W} y + x\odot_{\cal W} z - x =(v_{1},
v_{1}^{\prime})\odot_{\cal W}(v_{2}, v_{2}^{\prime}) + (v_{1},
v_{1}^{\prime})\odot_{\cal W}(v_{3}, v_{3}^{\prime}) -$\\
$-(v_{1}, v_{1}^{\prime})= ( v_{1}\odot_{V}v_{2},
v_{1}^{\prime}\odot_{V^{\prime}}v_{2}^{\prime} ) + (
v_{1}\odot_{V}v_{3},
v_{1}^{\prime}\odot_{V^{\prime}}v_{3}^{\prime} ) - (v_{1},
v_{1}^{\prime})=$ \\
$=( v_{1}\odot_{V}v_{2} +  v_{1}\odot_{V}v_{3} - v_{1},
v_{1}^{\prime}\odot_{V^{\prime}}v_{2}^{\prime}  +
v_{1}^{\prime}\odot_{V^{\prime}}v_{3}^{\prime}  - v_{1}^{\prime}
).$

Using now the relations (a), (b) and the relations $3.1.4(1)$ for
$V$ and $ V^{\prime} $, we obtain the equality $ x\odot_{\cal W}(
y + z - \beta_{\cal W}(x)) = x\odot_{\cal W} y + x\odot_{\cal W} z
- x.$ Hence the condition 3.1.4 (1) holds.

We verify now the relation $3.1.4 (4)$. For this, let $ x =
(v_{1}, v_{1}^{\prime} )\in {\cal W},~ y = (v_{2},
v_{2}^{\prime})\in {\cal W} $  such that $\alpha_{\cal W}(y) =
\beta_{\cal W}(x) $ and $ k \in K.$ We have

$(c)~~(k y + (1-k) \alpha_{\cal W}(x)) \odot _{\cal W}x =$\\
$= ( k v_{2} + ( 1-k)\alpha_{V}(v_{1}), k v_{2}^{\prime} + (
1-k)\alpha_{V^{\prime}}(v_{1}^{\prime}))\odot_{\cal W}(v_{1},
v_{1}^{\prime} ) =$\\
$= ( (k v_{2} + ( 1-k)\alpha_{V}(v_{1}))\odot_{V} v_{1}, ( k
v_{2}^{\prime}+
(1-k)\alpha_{V^{\prime}}(v_{1}^{\prime}))\odot_{V^{\prime}}v_{1}^{\prime}
)$ and

$(d)~~k (y \odot_{\cal W} x ) + (1-k)x =  k ( (v_{2},
v_{2}^{\prime})\odot_{\cal W}(v_{1}, v_{1}^{\prime} )) +
(1-k)(v_{1}, v_{1}^{\prime} ) =$\\
$= ( k( v_{2}\odot_{V} v_{1}) + (1-k)v_{1}, k (
v_{2}^{\prime}\odot_{V^{\prime}} v_{1}^{\prime}) +
(1-k)v_{1}^{\prime} )$.

Using the equalities (c) and (d) and the relations $3.1.4(4)$ for
$V$ and $ V^{\prime} $, we obtain that the condition 3.1.4 (4)
holds.

In the same manner we prove that the conditions 3.1.4 (2) and
3.1.4 (3) hold. Hence $ V\oplus V^{\prime} $ is a vector groupoid.

The vector groupoid $~( {\cal W}:= V\oplus V^{\prime},
\alpha_{\cal W}, \beta_{\cal W}, \odot_{\cal W}, \iota_{\cal W},
\Delta_{V_{0}})~$ is called the {\it Whitney sum}  of the vector
groupoids $~(V, V_{0})~$ and $~(V^{\prime}, V_{0}).$ The base of
this vector groupoid can be identified with $ V_{0}$.

\begin{prop}
{\it If $~(V, V_{0})~$ and $~(V^{\prime}, V_{0})$ are transitive
vector groupoids, then the Whitney sum $~(V\oplus V^{\prime},
\Delta_{V_{0}})~$ is a transitive vector groupoid.}
\end{prop}
\begin{proof}
It must prove that the anchor $ (\alpha_{\cal W}, \beta_{\cal W})
: {\cal W} \to \Delta_{V_{0}}\times \Delta_{V_{0}} $ is
surjective.
\end{proof}

 If $( V \oplus V^{\prime}, \Delta_{V_{0}})$ is the Whitney sum of
vector groupoids $ (V,V_{0}) $ and $ ( V^{\prime}, V_{0})$,  then
the projections maps $p: V \oplus V^{\prime}\to V $ and $
p^{\prime}: V\oplus V^{\prime}\to V^{\prime}$ defined by
$p(v,v^{\prime})= v $ and $ p^{\prime}(v ,v^{\prime})= v^{\prime}$
are morphisms of vector groupoids.

\begin{thm}
 {\it Let $ (V,V_{0}) $ and $ ( V^{\prime}, V_{0})$ be two vector groupoids.
 The triple $( V\oplus V^{\prime}, p, p^{\prime})$ verifies the {\bf universal
 property of the Whitney sum}:

for all triple $( U, q ,q^{\prime})$ composed by vector groupoid
$( U , \alpha_{U}, \beta_{U}, \odot_{U}, \iota_{U}, V_{0})$ and
two morphisms of vector groupoids $~
V^{\prime}~\stackrel{q^{\prime}}{\longleftarrow}~U~
\stackrel{q}{\longrightarrow}~V$,
 there exists a unique morphism of vector groupoids $~ \varphi: U~\to~V \oplus V^{\prime}~$ such that
the following diagram:\\[-0.3cm]
$$ V^{\prime}\stackrel{p^{\prime}}{\longleftarrow} V\oplus V^{\prime}\stackrel{p}{\longrightarrow} V$$\\[-0.9cm]
$$\stackrel{q^{\prime}}~{\nwarrow}~~~~{\uparrow}\varphi~~~~{\nearrow}q$$\\[-1cm]
$$U$$\\[-0.7cm]
is commutative.}
\end{thm}

\begin{proof}
 We consider the map $ \varphi: U \to
V \oplus V^{\prime}$ by taking $\varphi(u):=( q(u),q^{\prime}(u))$
for all $u\in U. $ By hypothesis the maps $ q: U \to V $ and $
q^{\prime} : U \to V^{\prime} $ are vector groupoid morphisms.
Then $ (\alpha_{V}\circ q )(u) = \alpha_{U}(u) $ and $
(\alpha_{V^{\prime}}\circ q^{\prime} )(u) = \alpha_{U}(u) $, for
all $ u\in U.$ It follows that $ \alpha_{V}( q (u) ) =
\alpha_{V^{\prime}}(q^{\prime} (u))$.  Similarly $ \beta_{V}( q
(u) ) = \beta_{V^{\prime}}(q^{\prime} (u))$. Therefore $
\varphi(u) \in W:= V\oplus V^{\prime}.$ Hence $ \varphi $ is
well-defined.

Let now $ x, y\in U$ such that $ (x,y)\in U_{(2)}, $ i.e. $
\beta_{U}(y)= \alpha_{U}(x).$ Also we have $ ( q(x), q(y))\in
V_{(2)}$, i.e. $ \beta_{V}(q(y))= \alpha_{V}(q(x)),$ since $ q $
is a groupoid morphism. Then $ (\varphi (x), \varphi(y))\in
W_{(2)}$. Indeed,  $ \beta_{W}(\varphi(y))= \beta_{W}(q(y),
q^{\prime}(y))= ( \beta_{V}(q(y)),\beta_{V}(q(y)))= (
\alpha_{V}(q(x)),\alpha_{V}(q(x)) =
\alpha_{W}(q(x),q^{\prime}(x))= \alpha_{W}(\varphi(x))$.

 For $ x, y\in U$ such that $ (x,y)\in U_{(2)} $ we have $~
\varphi(x \odot_{U} y) =$\\
$=( q (x \odot_{U} y), q^{\prime} (x \odot_{U} y)) = (
q(x)\odot_{V} q(y),q^{\prime}(x)\odot_{V^{\prime}} q^{\prime}(y))
= \varphi(x)\odot_{W}\varphi(y).$

Using the linearity of $ q $ and $ q^{\prime}$ it is easy to
verify that $ \varphi $ is a linear map. Therefore, $ \varphi $ is
a vector groupoid morphism. We have $~p\circ \varphi = q~$ and
$~p^{\prime}\circ \varphi = q^{\prime}.$

In a standard manner we prove that $\varphi$ is a unique morphism
of vector groupoids such that the above diagram is commutative.
 \end{proof}

\vspace*{0.5cm}

DEPARTMENT OF MATHEMATICS, WEST UNIVERSITY OF TIMI\c SOARA, Bd. V.
P{\^A}RVAN,nr.4, 1900, TIMI\c SOARA, ROMANIA\\
\hspace*{0.7cm} E-mail:vpoputa@yahoo.com; ivan@math.uvt.ro


\begin{thebibliography}{99}
\addcontentsline{toc}{chapter}{Bibliografie}

\bibitem{brandt} H. Brandt, $\ddot {U}ber$\textit{ eine Verallgemeinerung der
Gruppen-Begriffes.} Math. Ann., \textbf{96} (1926), 360--366. MR
1512323.

\bibitem{brow87} R. Brown, \textit{From groups to groupoids: a brief survey.}
Bull. London Math. Soc., \textbf{19} (1987), 113--134.

 \bibitem{brow88} R. Brown, {\it Topology : Geometric Account of General
Topology, Homotopy Types and the Fundamental Groupoid}. Hal.
Press, New York, 1988.

\bibitem{canwei} A. Cannas da Silva and A. Weinstein, \textit{Geometric Models for
Noncommutative Algebras.} Berkeley Mathematics Lectures,
\textbf{10}, Amer. Math. Soc., Providence, R.I., 1999.

\bibitem{cone94} A. Connes, \textit{Noncommutative Geometry.} Academic Press Inc.
San Diego, CA, 1994.

\bibitem{cdw87}  A. Coste, P. Dazord \& A. Weinstein, \emph{Groupoides symplectiques,} Publ. Dept. Math. Lyon, 2/A
(1987),1--62.

\bibitem{diaste04} A. P. S. Dias and I. Stewart, \emph{Symmetry groupoids and admissible vector fields for coupled cell networks,} J. London
Math. Soc., {\bf 69}(2004), 707--736. MR 2005j:37034.


\bibitem{dumoiv}B. Dumons and Gh. Ivan, \textit{Introduction \`{a} la th\'{e}orie
des groupo\"{\i}des.} Dept. Math. Univ. Poitiers ( France
),URA,C.N.R.S. D1322,\textbf{ 86}, 1994.

\bibitem{ehre50} C. Ehresmann, \textit{O\'{e}uvres compl\'{e}tes. Parties I.1,
I.2. Topologie alg\'{e}brique et g\'{e}ometrie
diff\'{e}rentielle}. Dunod, Paris,1950.

\bibitem{higg71} P. J. Higgins, {\it Notes on Categories and Groupoids}. Von
Nostrand Reinhold Mathematical Studies {\bf 32}, London,1971. MR
48:6288.

\bibitem{golust06} M. Golubitsky and I. Stewart, \emph{Nonlinear dynamics of networks: the groupoid formalism,} Bull. Amer. Math.
Soc., {\bf 43}(2006), no. 3, 305--364.

\bibitem{ivan02} Gh. Ivan, \emph{Algebraic constructions of Brandt groupoids,} Proceedings of the Algebra Symposium, " Babe\c s- Bolyai" University, Cluj-Napoca, (2002), 69-90.

\bibitem{john}  C.K. Johnson, \textit{ Crystallographic groups, groupoids  and
orbifolds.}  Workshop  on Orbifolds, Groupoids and Their
Applications.University of Wales, Bangor, September, 2000.

\bibitem{mackey63} G. W. Mackey, {\it Ergodic theory, groups theory and
differential geometry}. Proc. Nat. Acad. Sci., USA, \textbf{50}
(1963), 1184--1191.

\bibitem{vpop} V. Popu\c{t}a, {\it Some classes of Brandt Groupoids}. Sci. Bull. of "Politehnica" Univ. of Timi\c{s}oara, Tom 55(66), Fasc. 1, 2007 (50-54).

\bibitem{rams71} A. Ramsey, \textit{Virtual groups and group actions.} Adv. in
Math., \textbf{6} (1971), 253--322.

\bibitem{ramsren01} A. Ramsey and J. Renault, \textit{Groupoids in Analysis,
Geometry and Physics.} Contemporary Mathematics, \textbf{282}, AMS
Providence, RI, 2001.

\bibitem{renaul} J. Renault, \textit{A Groupoid Approach to C*-Algebras.} Lecture
Notes Series, \textbf{793}, Springer--Verlag, Berlin, Heidelberg,
New York, 1980.

\bibitem{stegopi03} I. Stewart, M. Golubitski and M. Pivato,  {\it Symmetry
groupoids and patterns of synchrony in coupled cell networks}.
Siam J. Applied Dynamical Systems, {\bf 2} (2003), no.4, 609--646.

\bibitem{ziv06} R.T. $\check{Z}$ivaljevic,\textit{ Groupoids in combinatorics - applications of
a theory of local symmetries.} Algebraic and geometric
combinatorics, 305--324. Contemporary Math.,  \textbf{423}, Amer.
Math. Soc., Providence, RI, 2006.

\bibitem{weigro} A. Weinstein, \textit{Groupoids: Unifying internal and external
symmetries.} Notices Amer. Math. Soc., \textbf{43} (1996),
744--752. MR 97f:20072.

\bibitem{west} J. J. Westman, {\it Harmonic analysis on groupoids}.
Pacific J. Math., \textbf{27}, (1968), 621--632.
\end{thebibliography}
\end{document}